\documentclass[10pt,a4paper]{article}
\usepackage[ngerman,english]{babel}
\usepackage[T1]{fontenc}
\usepackage[ansinew]{inputenc}
\usepackage[pdftex]{graphicx}
\usepackage{a4wide}
\usepackage[usenames]{color}
\usepackage[intlimits]{amsmath}
\usepackage{amsmath,amsthm,amsfonts,amssymb,mathrsfs,graphicx,ngerman}
\usepackage{enumerate}
\usepackage{tabularx}
\usepackage{mathtools}

\usepackage{ifpdf}

\usepackage{hyperref}
\definecolor{Black}{rgb}{0,0,0}
\definecolor{CiteColor}{rgb}{0.1,0.8,0.1}
\hypersetup{colorlinks=true,%
linkcolor=Black,%
citecolor=CiteColor,%
filecolor=Black,%
menucolor=Black,%
pagecolor=Black,%
urlcolor=Black}

\usepackage[left=3cm,right=2cm,top=2cm,bottom=3cm,includeheadfoot]{geometry}

\input{skripten.sty}

\begin{document}
\title{Global Strong Solution for Large Data\\ to the Hyperbolic Navier-Stokes Equation} 

  \author{Alexander Schöwe}
     \date{27.09.2014}
  \maketitle

\selectlanguage{english}
\begin{abstract}
\noindent We consider a hyperbolic quasilinear fluid model, that arises from a delayed version for the constitutive law for the deformation tensor in the incompressible Navier-Stokes equation. We prove the existence of global strong solutions for large data including decay rates in $\R^2$ and in the three dimensional special cases known from the classical Navier-Stokes equation. As a corollary we can derive from \cite{meinpaper} a global relaxation limit $\tau\to 0$ uniform in time. Furthermore we give an improved version of the regularity criterion of \cite{fanozawa}.
\end{abstract}

\section{Introduction}
Let $n\geq 2$ and $T,\tau,\mu>0$. In this note the fluid model
\begin{equation}\label{dgl}
  	\begin{aligned}
   \tau u_{tt}-\mu\Delta u+u_{t}+\nabla p+\tau\nabla p_{t}&=-(u\cdot\nabla)u-(\tau u_{t}\cdot\nabla)u-(\tau u\cdot\nabla)u_{t}\\
	\operatorname{div} u &= 0\\
	u(0,\cdot)=u_{0},\quad &u_{t}(0,\cdot)=u_{1}
	\end{aligned}
 \qquad
	\begin{aligned}
		\text{in } &(0,T)\times \R^{n} ,\\ 	\text{in } &(0,T)\times \R^{n},\\
		&\text{in } \R^{n}
	\end{aligned}
\end{equation}
for the velocity field $u=u(t,x):(0,T)\times\R^{n}\to \R^{n}$ and the pressure $p=p(t,x):(0,T)\times\R^{n}\to \R$, where $u_{0}$ and $u_{1}$ are given initial data, will be considered. This model arises from a delayed version for the constitutive law for the deformation tensor in the incompressible Navier-Stokes equation (compare \cite{lokal}).\\
The model has already been studied in \cite{lokal}, \cite{global} and \cite{meinpaper}, where the global existence for small data and a relaxation limit $\tau\to 0$  local in time for \eqref{dgl} was proven. In contrast to the Navier-Stokes equation the global well-posedness for large data even in $\R^2$ was still open. We will prove the global existence for large data and small parameter $\tau$ with the following idea. First we show that the local solution to the hyperbolic Navier-Stokes equation in $\R^2$ decays for a small time, since for small $\tau$ it stays close to the solution of the classical Navier-Stokes equation (due to \cite{meinpaper}), for which the decay is well known. With some work we are than able to show that the smallness condition for a modified version of the global existence theorem \cite[Theorem 6.1]{global} is fullfilled and obtain a global solution including decay rates.\\
This basic idea was already used in \cite{hagstromlorenz} for hyperbolic-parabolic coupled systems, but their theorem does not apply to our case. We will comment on other relating results and discuss other approaches to the problem in subsection 3.2, especially we improve the regularity criterion of \cite{fanozawa}.\\
Since our global existence theorem will be independent of $\tau\leq \tau_0$ we have an important consequence for the relaxation limit. We can modify the local result from \cite{meinpaper} and use the $\tau$-independent decay rates to get a global relaxation limit $\tau\to 0$ uniform in time. Therefore the classical Navier-Stokes equation and the hyperbolic Navier-Stokes equation behave similar. Formally, regarding the derivation of the hyperbolic Navier-Stokes equation, the result on the approximation seems to be not astonishing but one has to keep in mind the works \cite{dqr} and \cite{fsr}, where it was shown that delayed systems, that are formally close together, can behave differently. For example in \cite{fsr} it was shown, that an equation, coming from a Cattaneo type law, might not be exponentially stable, although the same system with a Fourier type law is.  In \cite{dqr} it is even shown, that formal high taylor expansions of the delayed term, can lead to ill-posedness. \\

\noindent The paper is a short overview, where details will be published later and is organized as follows. In section two we modify the global existence theorem for small data due to \cite{global} to get a $\tau$-independent smallness condition, which is of course necessary for the indicated proof for large data. In section three we first give some a priori estimates to handle all norms appearing in the smallness condition and then give a sketch of the proof for the global existence theorem for large data and small parameter $\tau$. We conclude this section with some remarks on the conditions on $\tau$ and relating results. Furthermore we improve the regularity criterion of \cite{fanozawa}. In the last section we will prove a global relaxation limit $\tau\to 0$ uniform in time. 

\section{Global Existence Theorem with $\tau$-independent smallness condition}
The global existence for small data in \cite{global} was proven with a method by Klainerman and Ponce, like it is for example described in \cite{rackebuch}. The proof uses convergence rates coming from the damped wave equation to show an a priori estimate for the solution to \eqref{dgl} that finally makes it possible to reiterate the application of the local existence theorem to obtain a global solution. Since we want to show a $\tau$-independent theorem, we first of all have to consider the $\tau-$dependence of the convergence rates for the damped wave equation. With a transformation to the $\tau$-dependent problem one can derive from \cite{matsumura} the following result that corresponds to \cite[Lemma 5.1]{global}, but in contrast to that is independent of $\tau$.

\begin{lemma}\label{taugedwelle}
Let $v$ be the solution to 
\begin{equation*}
  	\begin{aligned}
  	\tau &v_{tt}-\mu\Delta v+v_{t}=0\\
		&v(0,\cdot)=0,\quad v_{t}(0,\cdot)=v_{1}
	\end{aligned}
  \qquad
	\begin{aligned}
		\text{in } &(0,\infty)\times \R^{n},\\
		&\text{in } \R^{n},
	\end{aligned}
\end{equation*}
with $n\in\{2,3\}$ and $\tau\leq 1$. For $\alpha \in \N^{n}_{0}$, $j\in\{0,1\}$ and $0\leq \varepsilon\leq \frac{1}{2}$  there exists a constant $c$ independent of $\tau$, such that the following estimates hold:
\begin{align}\label{abfallratenfürglobalenexistenzsatzfürkleinedaten2}
	\|\nabla^{\alpha}\partial^{j}_{t}v(t,\cdot)\|_{2}\leq c\tau^{1-j}(1+t)^{-(\frac{|\alpha|}{2}+j)} \|v_{1}\|_{X_2},
\end{align}
\begin{align}\label{abfallratenfürglobalenexistenzsatzfürkleinedaten1}
	\|\nabla^{\alpha}\partial^{j}_{t}v(t,\cdot)\|_{2}\leq c\tau^{1-j}(1+t)^{-(\frac{n}{4}+\frac{|\alpha|}{2}+j)} \|v_{1}\|_{X_1},
\end{align}
\begin{align}\label{abfallratenfürglobalenexistenzsatzfürkleinedatenepsilon}
	\|\nabla^{\alpha}\partial^{j}_{t}v(t,\cdot)\|_{2}\leq c\tau^{1-j}(1+t)^{-\frac{n}{4}+\frac{n}{2}\varepsilon-\frac{|\alpha|}{2}-j} \|v_{1}\|_{X_{\frac{1}{1-\varepsilon}}},
\end{align}
where
	\[X_{k}:=\begin{cases}

  L^{2}\cap L^{k},  & \text{for }|\alpha|+j=0\\
  W^{|\alpha|,2}\cap L^{k}, & \text{for  }|\alpha|+j\geq1.
\end{cases}
\]
Furthermore it holds for $1\leq p \leq 2 \leq q \leq \infty$ with $\frac{1}{p}+\frac{1}{q}=1$ and $\delta>0$ 
\begin{align}\label{abfallratenfürglobalenexistenzsatzfürkleinedatenq}
	\|\nabla^{\alpha}\partial^{j}_{t}v(t,\cdot)\|_{q}\leq c\tau^{1-j}(1+t)^{-\bigl(\frac{n}{2}(1-\frac{2}{q})+\frac{|\alpha|}{2}+j\bigr)} \|v_{1})\|_{Y_q},
\end{align}
where
\[Y_q:=W^{m_q,p}\]
and 
\[m_q:=\lceil(1-\tfrac{2}{q})(2+n)+\delta\rceil+|\alpha| \equiv m_0 +|\alpha|. \]
\end{lemma}
\noindent Changing the energy to 
\begin{equation}\label{energie}
	E_{m}(t):=E_{m}\left(u(t)\right):=\frac{1}{2} \sum_{|\alpha|\leq m+1}(\tau^2 \|\nabla^{\alpha}u_{t}\|^{2}_{2}+\tau\mu\|\nabla^{\alpha}\nabla u\|^{2}_{2}+\varepsilon_2\|\nabla^{\alpha}u\|^{2}_{2})(t),
\end{equation}
for a $0< \varepsilon_2<\frac{1}{2} $, one can show the following $\tau$-independent high energy estimate corresponding to \cite[Theorem 4.1]{global}.
\begin{theorem}\label{highenergyestimate}
For $\tau\leq 1$ there exists a constant $c$, independent of the local existence time $T$, the data $(u_{0},u_{1})\in \left(W^{m+2,2}\cap L^{2}_{\sigma}\right)\times \left(W^{m+1,2}\cap L^{2}_{\sigma}\right)$ and $\tau$, such that for $0\leq t \leq T$ and $m>\tfrac{n}{2}+1$
\begin{equation}
	E_{m}(t)\leq c E_{m}(0)e^{c\int^{t}_{0}(\|u\|^{2}_{\infty}+\tau\|u_{t}\|_{1,\infty}+\|\nabla u\|_{\infty}+\tau^2\|u_{t}\|_{1,\infty}^2+\|\nabla u\|_{\infty}^2)(r)\mathrm d r}.
\end{equation}
\end{theorem}
\noindent Using the representation formula for the solution \cite[Theorem 4.1]{global}, Theorem \ref{highenergyestimate} and Lemma \ref{taugedwelle} one can prove the following $\tau$-independent version of \cite[Theorem 5.3]{global}.
\begin{theorem}\label{spezielleapriori}
Let $4<q<\infty$, $m_0$ from Lemma \ref{taugedwelle}, $m_,m_1\in \N$ with $m_1\geq 2$, $m\geq m_1+m_0+3$, $p:=\tfrac{q}{q-1}$, $0 \leq \varepsilon \leq \tfrac{1}{q}$ and $\tau\leq\tfrac{1}{2}$. There exists a $\delta_{1}>0$ independent of $\tau$ such that for initial data $(u_{0},u_{1})$ with
\begin{equation}\label{anfangsdelta}
	\|u_{0}\|_{m+4,2}+\tau\|u_{1}\|_{m+3,2}+\|u_{0}\|_{\frac{1}{1-\varepsilon}}+\tau\|u_{1}\|_{\frac{1}{1-\varepsilon}}+\|u_{0}\|_{m_1+m_0+2,p}+\tau\|u_{1}\|_{m_1+m_0+1,p}<\delta_{1},
\end{equation}
there exists a $M_{0}>0$ independent of $T$ and $\tau$, such that the solution $u$ satisfies
\begin{align}\label{MT}
\nonumber	M(T):=\sup_{0 \leq t \leq T}\Bigl\{(&1+t)^{1-\frac{2}{q}}\|u(t)\|_{m_1,q}+(1+t)^{\frac{3}{2}-\frac{2}{q}}(\tau\|u_{t}(t)\|_{m_1,q}+\|\nabla u(t)\|_{m_1,q})\\
&+(1+t)^{\frac{1}{2}-\varepsilon}\|u(t)\|_{m,2}+(1+t)^{1-\varepsilon}(\tau\|u_{t}(t)\|_{m,2}+\|\nabla u(t)\|_{m,2})\Bigr\}\leq M_{0}
\end{align}
\end{theorem}

\begin{remark}\
\vspace{-15pt}
 \begin{enumerate}
  \item It is not enough to just check the estimates in \cite{global} on their $\tau$-dependence. Therefore we put for example an artifical $\tau$ in the definition of the energy and $M(T)$.
\item The introduction of $\varepsilon$ is necessary for the proof of the global existence theorem for large data, since we can not show an a prioi estimate for the $L^1$-norm.
 \end{enumerate}
\end{remark}
\noindent Now one can derive a global existence theorem with a smallness condition and decay rates independent of $\tau$ as usual.
\begin{theorem}\label{globalelösungdeltatauunabhängig}
 Let $4<q<\infty$, $m_,m_1\in \N$ with $m_1\geq 2$, $m\geq m_1+5+n$, $p:=\tfrac{q}{q-1}$, $0 \leq \varepsilon \leq \tfrac{1}{q}$ and $\tau\leq 1$. There exists a $\delta>0$ independent of $\tau$ such that for initial data $(u_{0},u_{1})$ with
\begin{equation}
	\|u_{0}\|_{m+4,2}+\tau\|u_{1}\|_{m+3,2}+\|u_{0}\|_{\frac{1}{1-\varepsilon}}+\tau\|u_{1}\|_{\frac{1}{1-\varepsilon}}+\|u_{0}\|_{m_1+n+4,p}+\tau\|u_{1}\|_{m_1+n+3,p}<\delta,
\end{equation}
there exists a unique global solution $(u,p)$ to \eqref{dgl} satisfying
\begin{align}
	u\in C^{0}([0,\infty)&,W^{m+4,2}\cap L^{2}_{\sigma})\cap C^{1}([0,\infty),W^{m+3,2})\cap C^{2}([0,\infty),W^{m+2,2}),\\
&\nabla (p+\tau p_{t})\in C^{0}([0,\infty),W^{m+2,2}).
\end{align}
Furthermore it holds 
\begin{align*}
	\|u(t)\|_{m,2}=\mathcal{O}\left(t^{-(\frac{1}{2}-\varepsilon)}\right),\quad \tau\|u_{t}(t)\|_{m,2}=\mathcal{O}\left(t^{-(1-\varepsilon)}\right),\quad
\|\nabla u(t)\|_{m,2}&=\mathcal{O}\left(t^{-(1-\varepsilon)}\right)\\
	\|u(t)\|_{m_1,q}=\mathcal{O}\left(t^{-(1-\frac{2}{q})}\right),\;\, \tau\|u_{t}(t)\|_{m_1,q}=\mathcal{O}\left(t^{-(\frac{3}{2}-\frac{2}{q})}\right),\;\,
\|\nabla u(t)\|_{m_1,q}&=\mathcal{O}\left(t^{-(\frac{3}{2}-\frac{2}{q})}\right)
\quad \text{ for }t\to \infty.
\end{align*}
\end{theorem}

\section{Global Existence Theorem for Large Data}
\subsection{A Priori Estimates}
As written in the introduction one uses the decay of the solution to the classical Navier-Stokes equation together with \cite{meinpaper} to show a decay of the $H^s-$norms of the local solution to \eqref{dgl}.
\begin{lemma}\label{asmptotiknavierstokes}
Let $v_0\in W^{m,2}\cap L^2_\sigma\cap L^1$ and $m\geq 3$, then there is global solution to the two dimensional classical Navier-Stokes equation
\begin{equation}
 v\in C([0,\infty),W^{m,2}\cap L^2_\sigma)\cap C^1([0,\infty),W^{m-2,2}).
\end{equation}
Furthermore for $\alpha\in \N_0^2$ with $|\alpha|\leq m$ the following decay rates for a constant $c=c(\|v_0\|_{...})$ hold 
\begin{align}
 \|\nabla^\alpha v(t)\|_2&\leq c(1+t)^{-(\frac{|\alpha|}{2}+\frac{1}{2})},\\
 \|\nabla^\alpha v(t)\|_\infty&\leq c(1+t)^{-(\frac{|\alpha|}{2}+1)}.
\end{align}
Additional one has for $\alpha\in \N_0^2$ with $|\alpha|\leq m-2$ and a constant $c=c(\|v_0\|_{...})$ 
\begin{align}
\|\nabla^\alpha v_t(t)\|_2&\leq c(1+t)^{-(\frac{|\alpha|}{2}+\frac{3}{2})}.
\end{align} 
Especially it holds $v_0\in W^{m+3,2}\cap L^2_\sigma\cap L^1$ with $c_1=c_1(\|v_0\|_{...})$
\begin{equation}
	E_{m}(v(t)):=\frac{1}{2} \sum_{|\alpha|\leq m+1}(\tau \|\nabla^{\alpha}v_{t}\|^{2}_{2}+\mu\|\nabla^{\alpha}\nabla v\|^{2}_{2}+\|\nabla^{\alpha}v\|^{2}_{2})(t)\leq c_1 (1+t)^{-1}.
\end{equation}
\end{lemma}

\begin{proof}
Due to \cite[Theorem 1]{kato86}) there is a solution
\begin{equation}
 v\in C([0,\infty),W^{m,2}\cap L^2_\sigma)\cap C^1([0,\infty),W^{m-2,2}).
\end{equation}
Therefore the convergence rates follow from \cite[Theorem 3.2]{schonbek}.
\end{proof}

\noindent The other norms in the smallness condition of Theorem \ref{globalelösungdeltatauunabhängig} can be handled with the following lemma.
\begin{lemma}
 Let $\tau\leq 1$ and $v$ be the solution to
\begin{equation*}
  	\begin{aligned}
  	\tau &v_{tt}-\mu\Delta v+v_{t}=0\\
		&v(0,\cdot)=0,\quad v_{t}(0,\cdot)=v_{1}
	\end{aligned}
  \qquad
	\begin{aligned}
		\text{in } &(0,\infty)\times \R^{n},\\
		&\text{in } \R^{n},
	\end{aligned}
\end{equation*}
with initial data $v_1\in  W^{s,2}\cap L^\mathtt{p}$, where $1\leq \mathtt{p}   \leq 2$. Then for $\alpha \in \N^{n}_{0}$, $j\in \N_{0}$ with $0\leq |\alpha|+j\leq s$ and $j\neq s$ there exists $\tau-$independent constant $c$ such that
\begin{align}\label{abschätzungfürdennichtlinearenteil}
 \|\nabla^{\alpha}\partial^{j}_{t}v(t,\cdot)\|_{\mathtt{p}}\leq c \tau(\tau+t)^{-j-\frac{|\alpha|}{2} } (\|v_{1}\|_{\mathtt{p}}+\|v_{1}\|_{s,2}).
\end{align}
If additonally $v_1\in L^1$, then
\begin{align}\label{abschätzungfürdenlinearenteil}
 \|\nabla^{\alpha}\partial^{j}_{t}v(t,\cdot)\|_{\mathtt{p}}\leq c \tau(\tau+t)^{-j-\frac{|\alpha|}{2} -\frac{n}{2}(1-\frac{1}{\mathtt{p}})} (\|v_{1}\|_{1}+\|v_{1}\|_{s,2}).
\end{align}
 \end{lemma}
\begin{proof}
Use a transformation on the $\tau$-independent problem and the papers \cite{kono} and \cite{konoungerade}.
\end{proof}
\noindent Applying these decay rates to the representation formula for the solution \cite[Theorem 4.1]{global} one can show 
\begin{lemma}\label{asmptotikhyperbolischnavierstokeszusammengefasst}
 Let $\tau\leq 1$, $1\leq t$ and $s\in \N_0$. Furhtermore let $4<q<\infty$, $p:=\frac{q}{q-1}$, $0 < \varepsilon\leq \frac{1}{q}$ and $m_1\geq 2$. Then it holds the $\tau$-independent estimates
\begin{equation}
\label{ulespilonabschätzung}
\begin{aligned}
 \|u(t)\|_{m_1+6,p}+\|u(t)\|_{\frac{1}{1-\varepsilon}}\leq &c_2(\tau+t)^{-\varepsilon}(\|u_0\|_{m_1+8,2}+\|u_0\|_{1}+\tau\|u_1\|_{m_1+7,2}+\tau\|u_1\|_{1})\\&
 + c_2\int^{t}_{0}(\tau+t-r)^{-\frac{1}{2}}E_{m_1+7}(u(r))\dd r,\\
\end{aligned}
\end{equation}
\begin{equation}
\label{utlespilonabschätzung}
\begin{aligned}
\|u_t(t)\|_{m_1+5,p}+\|u_t(t)\|_{\frac{1}{1-\varepsilon}}\leq  &c_3(\tau+t)^{-(\varepsilon+1)}(\|u_0\|_{m_1+7,2}+\|u_0\|_{1}+\tau\|u_1\|_{m_1+7,2}+\tau\|u_1\|_{1}) \\&+ c_3\int^{t}_{0}(\tau+t-r)^{-\frac{3}{2}}E_{m_1+6}(u(r))\dd r.
\end{aligned}
\end{equation}
\end{lemma}

\subsection{The Theorem}
Now one can prove the following theorem for large inital data.
\begin{theorem} \label{globalelösunggrosseanfangsdaten}
Let $4<q<\infty$, $m_,m_1\in \N$ with $m_1\geq 2$, $m\geq m_1+7$ and $0<\varepsilon \leq \frac{1}{q}$. Then for $(u_{0},u_{1})\in \left(W^{m+7,2}(\R^{2})\cap L^{2}_{\sigma}(\R^{2})\cap L^{1}(\R^{2})\right)\times \left(W^{m+3,2}(\R^{2})\cap L^{2}_{\sigma}(\R^{2})\cap L^{1}(\R^{2})\right)$ and $\tau=\tau\bigl(\|u_0\|_{..},\|u_1\|_{..}\bigr)$ small enough, there exists a global solution $(u,p)$ to \eqref{dgl} with
\begin{align}
	u\in C^{0}([0,\infty)&,W^{m+4,2}\cap L^{2}_{\sigma})\cap C^{1}([0,\infty),W^{m+3,2})\cap C^{2}([0,\infty),W^{m+2,2}),\\
&\nabla (p+\tau p_{t})\in C^{0}([0,\infty),W^{m+2,2}).
\end{align}
Furhermore it holds 
\begin{align*}
	\|u(t)\|_{m,2}=\mathcal{O}\left(t^{-(\frac{1}{2}-\varepsilon)}\right),\quad \tau\|u_{t}(t)\|_{m,2}=\mathcal{O}\left(t^{-(1-\varepsilon)}\right),\quad
\|\nabla u(t)\|_{m,2}&=\mathcal{O}\left(t^{-(1-\varepsilon)}\right)\\
	\|u(t)\|_{m_1,q}=\mathcal{O}\left(t^{-(1-\frac{2}{q})}\right),\;\, \tau\|u_{t}(t)\|_{m_1,q}=\mathcal{O}\left(t^{-(\frac{3}{2}-\frac{2}{q})}\right),\;\,
\|\nabla u(t)\|_{m_1,q}&=\mathcal{O}\left(t^{-(\frac{3}{2}-\frac{2}{q})}\right)
\quad \text{ for }t\to \infty.
\end{align*}
\end{theorem}
\begin{beweisskizze}
 The proof consists of five steps:
\begin{enumerate}
 \item Impose conditions (inspired by Lemma \ref{asmptotiknavierstokes} and Lemma \ref{asmptotikhyperbolischnavierstokeszusammengefasst}) on a certain point in time $T^*$ (we want to apply the
global existence theorem for small data at $T^*$).
\item Impose conditions on $\tau$ (inspired by \cite[Theorem 3.2]{meinpaper}).
\item Show that the local solution can be continued until $T^*$. 
\item Prove with Lemma \ref{asmptotiknavierstokes} and Lemma \ref{asmptotikhyperbolischnavierstokeszusammengefasst} that the smallness condition is fullilled at $T^*$.
\item Apply the Global Existence Theorem for Small Data at $T^*$.
\end{enumerate}
\end{beweisskizze}
\begin{remark}
 Of course the prove extends to the three dimensional special cases, where the global strong solvability for large data of the classical Navier-Stokes equation is known. 
\end{remark}

\subsection{Condition on $\tau$}
The permitted size of the relaxation parameter $\tau$ in the global existence theorem \ref{globalelösunggrosseanfangsdaten} depends on the size of the initial data and therefore the general question of solution to large data remains open. On the other hand in applications one usually thinks of small relaxations and furthermore there are quite a few similar results.\\
For example in the theorems \cite[Theorem 2.2]{bnp}, \cite[Theorem 0.1]{pr} und \cite[Theorem 1]{hach}, where the equation
\begin{equation}\label{paicugleichung}
  	\begin{aligned}
   \tau u_{tt}+u_{t}+(u\cdot\nabla)u-\mu\Delta u+\nabla p&=0\\
\divergenz u&=0 \\
		u(0,\cdot)=u_{0},\quad &u_{t}(0,\cdot)=u_{1}
	\end{aligned}
 \qquad
	\begin{aligned}
		\text{in } &(0,T)\times \R^{n},\\
\text{in } &(0,T)\times \R^{n}, \\
		&\text{in } \R^{n},
	\end{aligned}
\end{equation} 
is studied, one has a similar constraint on $\tau$. In the paper \cite{hagstromlorenz} this is also the case.\\
Another comparison can be found in \cite{finitepropagation} where the equation 
\begin{equation}
  	\begin{aligned}
   \tau u_{tt}+u_{t}+(u\cdot\nabla)u-\mu\Delta u&-\tfrac{1}{\alpha}\nabla \divergenz u=0\\
		u(0,\cdot)=u_{0},\quad &u_{t}(0,\cdot)=u_{1}
	\end{aligned}
 \qquad
	\begin{aligned}
		\text{in } &(0,T)\times \R^{n},\\
		&\text{in } \R^{n},
	\end{aligned}
\end{equation} 
is studied. The global existence theorem \cite[Theorem 2]{finitepropagation} imposes analogue conditions on $\tau$ and $\alpha$.\\

\noindent In a forthcoming paper we will show that the same holds true for the minor changed model
\begin{equation}
  	\begin{aligned}
   \tau u_{tt}+u_{t}+(u\cdot\nabla)u+\divergenz u\, u-\mu\Delta u&-\tfrac{1}{\alpha}\nabla \divergenz u=0\\
		u(0,\cdot)=u_{0},\quad &u_{t}(0,\cdot)=u_{1}
	\end{aligned}
 \qquad
	\begin{aligned}
		\text{in } &(0,T)\times \R^{n},\\
		&\text{in } \R^{n},
	\end{aligned}
\end{equation} 
but this model also has a blow-up if the conditions on the parameters and the data are not satisfied. This shows that the question of global solutions to large data can be very delicate.\\

\noindent For the classical Navier-Stokes equation one has the very famous Beale-Kato-Majda criterion \cite{bealekatomajda} which states that under the condition 
\begin{equation}
 \operatorname{rot} u\in L^1((0,\infty),L^\infty)
\end{equation}
there is no blow-up in finite time. This criterion has been improved by \cite{kozonoogawataniuchi} to 
\begin{equation}
 \operatorname{rot} u\in L^1((0,\infty),\dot{B}^0_{\infty,\infty}),
\end{equation}
where $\dot{B}^0_{\infty,\infty}$ stands for the homogenous Besov space.\\
For the hyperbolic Navier-Stokes equation it is proven in \cite{fanozawa} that under the condition 
\begin{equation}
 u,\nabla u,u_t\in L^1((0,T),\dot{B}^0_{\infty,\infty}),
\end{equation}
there is no blow-up. Using the representation 
\begin{equation}
  	\begin{aligned}
   u_{t}(t)+P\left((u\cdot\nabla)u\right)(t)-\frac{\mu}{\tau}\int_0^t e^{-\tfrac{(t-s)}{\tau}}P\Delta u(s)\dd s&=e^{-\frac{t}{\tau}}(u_1+P(u_0\cdot \nabla)u_0) \qquad &\text{in } (0,T)\times \R^{n},\\
   		u(0,\cdot)&=u_{0} \qquad &\text{in } \R^{n}
	\end{aligned}
 \end{equation} 
for $u\in L^2_\sigma(\R^n)$ of \eqref{dgl}, one can improve this to the following theorem, but nevertheless even in $\R^2$ it seems to be not possible to show that the regularity criterion is satisfied, as one can easily do for the classical Navier-Stokes equation.
\begin{theorem}\label{satzregularitätskriterium}
 Let $s>m>\frac{n}{2}$ and $(u_{0},u_{1})\in \left(W^{s+2,2}(\R^{n})\cap L^{2}_{\sigma}(\R^{n})\right)\times \left(W^{s+1,2}(\R^{n})\cap L^{2}_{\sigma}(\R^{3})\right)$ and 
\begin{align*}u\in C^{0}([0,T),W^{s+2,2}(\R^{n})\cap L^{2}_{\sigma}(\R^{n}))\cap C^{1}([0,T),W^{s+1,2}(\R^{n}))\cap C^{2}([0,T),W^{s,2}(\R^{n}))
\end{align*}
the local solution due to \cite[Theorem 2.1]{meinpaper}. If
\begin{equation}
 \operatorname{rot} u\in L^1((0,T),\dot{B}^0_{\infty,\infty}),
\end{equation}
then one can continue the solution $u$ beyond $T$, which means there is a $T^*>T$, such that 
\begin{align*}u\in C^{0}([0,T^*),W^{s+2,2}(\R^{n})\cap L^{2}_{\sigma}(\R^{n}))\cap C^{1}([0,T^*),W^{s+1,2}(\R^{n}))\cap C^{2}([0,T^*),W^{s,2}(\R^{n})).
\end{align*}
\end{theorem}

\section{Global Relaxation Limit $\tau\to 0$}
\noindent If one changes the estimates (84), (88), (89) and (91) in the proof of \cite[Theorem 3.2]{meinpaper} and uses the $\tau$-independent decay rates one easily obtains the following result.
\begin{theorem}
Let $u$ be the $\tau-$independent global solution to \eqref{dgl} in one of the following cases:
\begin{enumerate}[(i)]
\item to large initial data and small parameter $\tau$ in $\R^2$ due to Theorem \ref{globalelösunggrosseanfangsdaten}, 
 \item to small initial data $\R^2$ or $\R^3$ due to Theorem \ref{globalelösungdeltatauunabhängig} (let additonally $u_0 \in W^{m+5,2}$ hold),
\item to small data as necessary for three dimensional classical Navier-Stokes (they are much weaker
than in (ii)) and small parameter $\tau$ due to Theorem \ref{globalelösunggrosseanfangsdaten}
\item to small 3d-perturbations of two dimensional initial data $\R^3$ due to Theorem \ref{globalelösunggrosseanfangsdaten},
\item to axially symmetric initial data or a small pertubations of that in $\R^3$ due to Theorem \ref{globalelösunggrosseanfangsdaten}.
\end{enumerate}
Then the following uniform estimates in $t$ hold 
\begin{equation}
\forall t\in[0,\infty):\qquad	\|u^{\tau}(t)-v(t)\|_{m+2,2}\leq c\tau \quad \text{ and } \quad	\|u^{\tau}_{t}(t)-v_{t}(t)\|_{m+1,2}\leq c\sqrt{\tau},
\end{equation}
where $v$ denotes the solution to the classical Navier-Stokes equation.
\end{theorem}
\begin{remark}
 In contrast to the papers \cite{pr}, \cite{hach} and \cite{bnp} this is a uniform estimate in $t$. 
\end{remark}

\bibliographystyle{alpha}

\bibliography{literatur}

\textsc{University of Konstanz, Department of Mathematics and Statistics, 78464 Konstanz, Germany}
\end{document}